\documentclass[smallextended,numbook,runningheads]{svjour3} 
\usepackage{amsfonts,latexsym,amssymb}
\usepackage{graphicx,amsmath}
\usepackage{algorithm}	
\usepackage{algorithmic}	
\usepackage{multirow}
\usepackage{caption}
\usepackage{subcaption}
\newcommand{\R}{{\mathbb R}}
\newcommand{\C}{{\mathbb C}}

\newcommand{\bx}{\mbox{\boldmath{$x$}}}

\newcommand{\be}{\mbox{\boldmath{$e$}}}

\newcommand{\bg}{\mbox{\boldmath{$g$}}}

\newcommand{\bl}{\begin{list}{ \ }{
\leftmargin=.325in}}
\newcommand{\el}{\end{list}}

\newcommand{\conj}[1]{\mbox{$\overline{#1}$}}
\newcommand{\inner}[1]{\mbox{$\left\langle #1 \right\rangle$}}
\newcommand{\norm}[1]{\mbox{$\left\Vert#1\right\Vert$}}
\newcommand{\abs}[1]{\mbox{$\vert#1\vert$}}
\newcommand{\bignorm}[1]{\mbox{$\left\Vert#1\right\Vert$}}
\newcommand{\bigabs}[1]{\mbox{$\left\vert#1\right\vert$}}
\usepackage{booktabs}
\usepackage{scalerel}
\usepackage{xcolor}%
\definecolor{SPECgreen}{rgb}{0,0.75,0}%
\definecolor{SPECorange}{rgb}{1.0,.5625,0}%
\definecolor{SPECred}{rgb}{0.75,0,0}

\definecolor{SPECblue}{rgb}{0,0,0.75}
\definecolor{SPECblack}{rgb}{0.25,0.25,0.25}
\definecolor{SPECgray}{rgb}{0.625,0.625,0.625}
\definecolor{citecol}{rgb}{0.75,0,0}
\definecolor{linkcol}{rgb}{0,0,0.75}
\definecolor{SPECpink}{rgb}{1.0,0.1,0.6}
\definecolor{SPECorangeg}{rgb}{0.609375,.609375,0.609375}
\definecolor{SPECblueg}{rgb}{0.0525,0.0525,0.0525}
\definecolor{SPECredg}{rgb}{0.0525,0.0525,0.0525}
\usepackage{tikz,pgfplots}%
\usetikzlibrary{calc,arrows,decorations.markings,patterns}%
\pgfplotsset{compat=1.14}
\usepackage{enumitem} 

\numberwithin{equation}{section}
  
\begin{document}

\date{\today}%
\title{Continuous Adaptive Cross Approximation for Ill-posed Problems with
   Chebfun }%
\author{%
  A. Alqahtani\thanks{Department of Mathematics, King Khalid University,
    P.O. Box 9004, Abha, Saudi Arabia and Department of Mathematical Sciences,
    Kent State University, Kent, OH 44242, USA. E-mail:
    \texttt{aalqah11@gmail.com}.}%
  \and%
  T. Mach\thanks{Department of Mathematical Sciences, Kent State University,
    Kent, OH 44242, USA. E-mail: \texttt{thomas.mach@gmail.com}.}%
  \and%
  L. Reichel\thanks{Department of Mathematical Sciences, Kent State University,
    Kent, OH 44242, USA. E-mail: \texttt{reichel@math.kent.edu}.}%
}%
\institute{}%
\maketitle

\begin{abstract}
The analysis of linear ill-posed problems often is carried out in function
spaces using tools from functional analysis. However, the numerical solution 
of these problems typically is computed by first discretizing the problem and
then applying tools from (finite-dimensional) linear algebra. The present paper 
explores the feasibility of applying the Chebfun package to solve ill-posed 
problems. This approach allows a user to work with functions instead of matrices.
The solution process therefore is much closer to the analysis of ill-posed
problems than standard linear algebra-based solution methods.
\end{abstract}



\section{Introduction}\label{sec:intro}
We are interested in the solution of Fredholm integral equations of the first kind,
\begin{equation}\label{inteq} %
  \int_{\Omega_{1}} \kappa(s,t) x(t)\,dt = g(s),\qquad s\in\Omega_{2},
\end{equation} %
with a square integrable kernel $\kappa$.
\ The $\Omega_{i}$ are subsets of $\R^{d_{i}}$ for
$i=1,2$. Such integral equations are common in numerous applications including
remote sensing, computerized tomography, and image restoration.

Two major problems arise when solving \eqref{inteq}. The first problem is that
the space of functions is of infinite dimensionality. The second problem is that
small changes in $g$ may correspond to large changes in $x$ as exemplified by 
\begin{equation*}
	\begin{split} 
	\max_{s\in \Omega_2} \abs{\int_{\Omega_1} \kappa(s,t)\cos(2\pi m t)\,dt}, \qquad \Omega_1=\Omega_2=[0,1],
	\end{split}
\end{equation*}
where the maximum can be made tiny by choosing $m$ large, despite the maximum of
$\abs{\cos(2\pi m t)}$ being $1$. This is a consequence of the Riemann--Lebesgue
theorem; see, e.g., \cite{Eng2000,Han1998} or below for discussions of this result. The 
second problem 
is particularly relevant when the right-hand side $g$ is a measured quantity subject to 
observational errors, as is the case in many applications.

Usually one deals with the first problem by first discretizing the functions $x(t)$ and
$g(s)$ in \eqref{inteq} using $n$ piecewise constant, linear, or polynomial basis 
functions; see e.g.,~\cite{Hac95} or~\cite{PCH4}. The kernel $\kappa(s,t)$ is discretized
analogously. This transforms the problem into a system of
linear equations. The second problem causes the coefficient matrix of said
system to be ill-conditioned for sufficiently large $n$. Straightforward
solution of these linear systems of equations generally is not meaningful because of severe
error propagation. Therefore, this linear system has to be regularized. This can, for
instance, be achieved by Tikhonov regularization or truncated singular value
decomposition (TSVD). While the first dampens the influence of the small
singular values, the latter outright ignores them. One is then often faced with
a trade-off between a small discretization error and a small error caused by the
regularization; see, e.g., Natterer \cite{Na77}. In fact, often the more basis
functions are used for the discretization, the more ill-conditioned the resulting
coefficient matrix becomes, and the larger the need of regularization.

In this paper we will first regularize the problem and then discretize the 
problem. Regularization will be achieved through a singular value expansion 
of the kernel. 
At the same time the singular value expansion provides us with an excellent
basis for discretizing the problem. The discretized problem is then a diagonal
linear system of equations, which can be solved trivially. Thus, dealing with
the second problem first simplifies the other problem.

We will compute the singular value expansion of the kernel using
Chebfun~\cite{chebfun}. Hence, our discretization basis will consist of
piecewise Chebyshev polynomials. The computed solution is a Chebfun
approximation to the function $x(t)$. The advantage of Chebfun is that the
solution will feel and behave like a function. Therefore, our approach is
arguably closer to directly solving \eqref{inteq} instead of a discretized
version.

This paper is organized as follows. In the second section, we will provide basic
definitions, introduce our notation, and briefly discuss Chebfun and singular
value expansions. Section \ref{sec:tsve} discusses the
truncated singular value expansion method (TSVE). The Tikhonov regularization method is described in Section \ref{sec:tik}. Numerical results that illustrates the performances of the methods of Sections \ref{sec:tsve} and \ref{sec:tik} are reported in Section \ref{sec:numerical_experiments}. Concluding remarks can be found in Section \ref{sec:concl}. 

\section{Basics}\label{sec:basics}

Let $L^{2}(\Omega_{i})$ for $i=1,2$ be spaces of Lebesgue measurable square
integrable functions with inner products
\begin{equation}\label{inner} %
  \inner{a(t),b(t)}_{\Omega_{i}} = \int_{\Omega_{i}} {\conj{a(t)}}\, b(t)\,dt, \quad \text{
  for } i=1,2,
\end{equation} %
where $\conj{a}$ represents the complex conjugate of
$a\in\C$. 
Based on these inner products, we can define $L^{2}$-norms by
\begin{equation*} %
  \norm{f(t)}_{\Omega_{i}}^{2}=\int_{\Omega_{i}} \abs{f(t)}^{2}\,dt, \quad \text{
  for } i=1,2.
\end{equation*} %
Throughout this paper $\norm{\cdot}$ stands for an $L^{2}$-norm. We will omit
the subscript if the domain is clear from the context. Since the spaces
$H_{i}:=L^{2}(\Omega_{i})$ for $i\in\{1,2\}$, with the inner products and norms
defined above, are complete vector spaces, they are \emph{Hilbert space}; see, e.g.,
\cite{Hac95}.

A given kernel $\kappa(\cdot,\cdot) \in L^{2}(\Omega_{1} \times \Omega_{2} )$
induces the bounded linear operator \cite[Thm.~3.2.7]{Hac95} 
$A : L^{2}(\Omega_{1})\rightarrow L^{2}(\Omega_{2})$ or
$H_{1} \rightarrow H_{2}$ defined by 
\begin{equation}\label{definition_A} %
  (Ax)(s)=\int_{\Omega_{1}} \kappa(s,t)x(t)\,dt.
\end{equation} %
The operator is sometimes called a \emph{Hilbert-Schmidt integral operator} and
the kernel $\kappa$ a \emph{Hilbert-Schmidt kernel}.  This allows us to write
\eqref{inteq} as
\begin{equation}\label{Axg}%
  Ax = g.
\end{equation} 
In particular, we assume that $g$ is in the range of $A$. Generally, we are 
interested in the solution of \eqref{Axg} of minimal norm. We refer to this
solution as $x_{\text{exact}}$.

In practice, the right-hand side $g$ of \eqref{inteq} is often a measured
quantity and therefore is subject to observational errors. Thus, we assume that
the error-free function $g$ is not available--only an error contaminated approximation
$g^{\delta}\in H_{2}$ of $g$ is known. We assume that $g^{\delta}$ satisfies %
\begin{equation*} 
  \norm{g-g^{\delta}} \leq \delta,
\end{equation*} %
with a \emph{known} bound $\delta>0$. The solution of the equation %
\begin{equation} \label{Axgd} %
  Ax=g^{\delta}, \quad \text{ with } x \in H_{1} \text{ and } g^{\delta} \in
  H_{2},
\end{equation} %
is generally not a meaningful approximation of the desired solution
$x_{\text{exact}}$ of \eqref{Axg}, since $A$ is not continuously invertible. 
In fact, the equation \eqref{Axgd} might not have a solution. 

The operator $A$ depends on the kernel $\kappa$. We will now have a closer look
at known theory about the kernel function $\kappa$. 
For any square integrable kernel $\kappa$, we define the singular value
expansion (SVE) \cite[\S4]{Schmidt1989} as
\begin{equation} \label{sve} %
  \kappa(s,t)=\sum_{i} \sigma_{i} \phi_{i}(s)\psi_{i}(t).
\end{equation}
The functions $\phi_{i}(s)$ and $\psi_{i}(t)$ are referred to as the 
singular functions. These functions are
orthonormal with respect to the usual inner product \eqref{inner}
\cite[\S5]{Schmidt1989}, i.e., %
\begin{align*}
  \inner{\psi_{i},\psi_{j}}_{\Omega_{1}} =
  \inner{\phi_{i},\phi_{j}}_{\Omega_{2}} 
  = \delta_{ij}, \quad \text{ with }i,j=1,2,\dotsc~.
\end{align*} %
The quantities $\sigma_{i}$ are known as singular values. It can be shown that
the only limit point of the singular values for square integrable kernels is
zero \cite[\S5]{Schmidt1989}.\footnote{\!\!\!Schmidt calls the singular values
  eigenvalues, since he is mainly concerned with symmetric kernels and the
  concept of singular values was not developed when he published his paper. We
  follow modern notation here.}  The singular values form a non-increasing
sequence: %
\begin{align*}
  \sigma_{1} \geq \sigma_{2} \geq \sigma_{3} \geq  \dotsb \geq 0. %
\end{align*} 
Let $\sum_{i=1}^{\infty}\sigma_{i}\phi_{i}(s)\psi(t)$ be a uniformly convergent
series. Then 
\begin{equation}\label{eq:SVE}
  \kappa(s,t)=\sum_{i=1}^{\infty}\sigma_{i}\phi_{i}(s)\psi(t),
\end{equation} %
as shown in \cite[\S8]{Schmidt1989}.  When the summation is finite, then the
kernel $\kappa(s,t)$ is said to be \emph{separable} (or \emph{degenerate}). Most
applications do not have a separable kernel. However, if the kernel is square
integrable, then it can be approximated well by a separable kernel with a 
suitable number of terms, $\ell$, in \eqref{eq:SVE}.  Let 
\begin{equation}\label{Aell} %
  \kappa_{\ell}=\sum_{i=1}^{\ell}\sigma_{i} \phi_{i}(s)\psi_{i}(t),
\end{equation} %
with the same ordering of the singular values. Then this is the closest kernel
of rank at most $\ell$ to $\kappa$ in the $L^{2}$-norm \cite[\S18 Approximation
Theorem]{Schmidt1989}. We will use this result to justify the application of the
truncated singular value expansion method (TSVE), which will be discussed in
Section~\ref{sec:tsve}.

We will also be using this Approximation Theorem to generally restrict our
expansion to singular values greater than $\varepsilon$, where $\varepsilon$ is a
small enough cut-off--say $10^{-8}$ or $10^{-16}$. Here, there is a trade-off between
computing time and approximation accuracy. We try to choose
$\varepsilon$ far below the regularization error so that it does
not have a significant effect on the accuracy. At the same time, a small
$\varepsilon$ means higher cost for computing the singular value
expansion and forming the computed approximate solution.

In this paper we will use two regularization methods, TSVE and Tikhonov
regularization. 
The TSVE method is based on the Approximation Theorem mentioned above. We approximate the
kernel $\kappa$ by $\kappa_{\ell}$ for some suitable $\ell\geq 0$. This results
in an approximation $A_{\ell}$ to $A$ and a solution, denoted by $x_{\ell}$, of
the problem
\begin{equation} \label{tsvd} %
  (A_{\ell}x)(s)=\int_{\Omega_{1}} \kappa_{\ell}(s,t) x(t)\,dt = g^{\delta}(s),
  \quad s\in\Omega_2.
\end{equation} %
The parameter $\ell$ is a regularization parameter that determines how many
singular values and basis functions of $\kappa$ are used to compute the
approximate solution $x_{\ell}$ of~\eqref{Axgd}. The remaining singular values,
which are smaller than or equal to $\sigma_{\ell}$,
are ignored. The singular value $\sigma_{\ell}$
provides information on the approximation error.

Tikhonov regularization replaces the system \eqref{Axgd} by the 
penalized least-squares problem
\begin{equation}\label{tikhonov}
  \min_{x\in H_{1}}\{\norm{Ax-g^{\delta}}^{2}+\lambda^{2}\norm{x}^{2}\},
\end{equation}
which has a unique solution $x_{\lambda}$ for any positive value of the
regularization parameter $\lambda$. Substituting the SVE \eqref{sve} into
\eqref{tikhonov} shows that Tikhonov regularization dampens the contributions to
$x_{\lambda}$ of singular values and functions with large index $k$ the most;
increasing $\lambda>0$ results in more damping. Since we cannot
deal with an infinite series expansion, we will, in practice, first
cut-off all singular values that are less than $\varepsilon$ as explained above, and then
apply Tikhonov regularization.

The determination of suitable values of the regularization parameters, $\ell$ in
\eqref{tsvd} and $\lambda$ in \eqref{tikhonov}, is important for the quality of
the computed approximate solution. Several methods have been described in the
literature including the discrepancy principle, the L-curve criterion, and
generalized cross validation; see \cite{BRS,Ki,KR,RR} for recent discussions of
their properties and illustrations of their performance.  Regularization methods
typically require that regularized solutions for several parameter values be
computed and compared in order to determine a suitable value.

\subsection{Chebfun}\label{subsec:chebfun}
We solve \eqref{inteq} by first regularizing followed by discretization.
However, we still want to compute the solution numerically. Thus, we need a
numerical library that can handle functions in an efficient way. Since a
function is representing uncountable many pairs of $x$ and $f(x)$ with
$x\mapsto f(x)$, a computer can only handle approximations to functions
numerically.\footnote{\!\!\!There are some notable exceptions like $\sin(x)$ or
  $x^{2}$. However, we cannot assume that the solution of \eqref{inteq} will
  fall into this very small set of functions.}

We chose the Matlab package Chebfun \cite{chebfun} for this purpose.  Chebfun
uses piecewise Chebyshev polynomials, so called chebfuns, to approximate
functions. All computations within Chebfun's framework are done with these
approximations to the actual function. This in turn means that we project the
functions $g\in L^{2}(\Omega_{2})$ onto a space of piecewise Chebyshev
polynomials over~$\Omega_{2}$. One can argue that this is a discretization.
However, Chebfun's framework is significantly different from other
discretizations in the sense that it gives the user the feeling of computing
with functions.

Chebfun's functionality includes the computation of sums and products of
functions and derivatives, inner products, norms, and integrals. Chebfun2/3,
Chebfun's extension to functions of two and three variables, can also compute
outer products and, most importantly for us here, the singular value expansion
\cite{TowTre13}.  The algorithm behind the singular value expansion uses a
continuous analogue of adaptive cross approximation. This is where some of the
motivation for this work originates, since we recently analyzed the application
of adaptive cross approximation to the solution of ill-posed problems \cite{MaRevBVa15}.

The approximation of $\kappa(s,t)$ is computed by an iterative process. First,
an approximation of the maximum point $(x,y)$ of $\kappa(s,t)$ is determined. The computation
of the exact maximum point is not important. The function is then approximated by
\begin{align*}
 \kappa_{1}(s,t) = \frac{\kappa(s,y) \kappa(x,t)}{\kappa(s,t)},
\end{align*} where $\kappa(s,y)$ and $\kappa(x,t)$ are one-dimensional chebfuns
in $s$ and $t$, respectively.

This process is then repeated for $\kappa(s,t) - \kappa_{1}(s,t)$ to find a
rank-1 approximation of the remainder. By recursion one obtains after $k$ steps
a rank-$k$ approximation to the original kernel. As soon as the remainder is
sufficiently small, the computed rank-$k$ approximation is the sought approximation
to $\kappa(s,t)$. At the end we have $\kappa(s,t)\approx C(s)D R(t)^{T}$, with
$C(s)$ and $R(t)$ vectors of functions, and $D$ a dense matrix of size
$k\times k$.

Based on this approximation it is easy to compute the singular value
expansion. Chebfuns continous analogue of the QR factorization can be used to
find orthogonal bases for $C(s)$ and $R(t)$. The upper triangular matrices are
multiplied by $D$ to form a new matrix $\tilde{D}$. Then a singular value
decomposition of $D=U\Sigma V^{T}$ is computed. Finally, the small orthogonal
matrices $U$ and $V^{T}$ are combined with $C(s)$ and $R(t)$, respectively; see
\cite{TowTre13}. A very similar process, called adaptive cross approximation
\cite{q467,q699}, was used in \cite{MaRevBVa15} for the discrete case of
matrices and vectors.

Chebfun has some limitations. Currently only functions of at most three variables
can be approximated by Chebfun. Hence, we are limited to ill-posed problems in one
space-dimension, and to problems in two space-dimensions for which the kernel is
separable and also given in a separable representation. This is the case for the
kernel that models Gaussian blur in two space-dimensions, making Gaussian blur our only 
example in two space-dimensions in this paper.

Chebfun2 and Chebfun3 are further limited to domains that are tensor products of
intervals. Thus, in this paper all domains are rectangles or rectangular
boxes. Chebfun also needs multivariate functions to be of low rank for an
efficient approximation, that is there has to exist a sufficiently accurate
separable approximation. This is for instance not the case for the kernel
$\kappa(s,t) = st - \min(s,t)$ from the deriv2 example of the Regularization
Tools package \cite{PCH4}. This limits the application of the methods described
in this paper. However, the Chebfun package is still under development and
some of the limitations mentioned might not apply to future releases.

\section{The TSVE method}\label{sec:tsve}

Assume that the kernel is non-separable and can be expressed as
\begin{equation}\label{kernel}
  \kappa(s,t)=\sum_{i=1}^{\infty} \sigma_{i} \phi_{i}(s)\psi_{i}(t),
\end{equation} 
and that the solution can be written as
\begin{equation}\label{solexpress}
  x(t)=\sum_{j=1}^{\infty} \beta_{j} \psi_{j}(t).
\end{equation}
The fact that $\kappa$ is non-separable implies that all $\sigma_i$ are positive,
and the assumption that the solution is of the form \eqref{solexpress} essentially 
states that the
solution has no component in the null space of $A$. This assumption is justified since the
null space of $A$ is orthogonal to all the $\psi_{j}$ and, thus, a component in
the direction of the null space would increase the norm of the solution, but not
help with the approximation of \eqref{inteq}. 
  
Substituting \eqref{kernel} and \eqref{solexpress} into \eqref{inteq}, and using
the orthonormality of the basis functions yields
\begin{align*}
  \sum_{i=1}^{\infty} \sigma_{i}\beta_{i}\phi_{i}(s)=g(s).
\end{align*} We further probe the equation with $\phi_{k}(s)$ for all $k$ and use the
orthonormality of the basis functions to obtain
\begin{align*}
  \sigma_{k}\beta_{k}=\int_{\Omega_{2}} \phi_{k}(s) g(s)\,ds,\qquad\forall\,  k. 
\end{align*} 
Thus, the exact solution to \eqref{Axg} is given by
\begin{equation}\label{truesol1D}
  x(t)=\sum_{j=1}^{\infty} \beta_{j} \psi_{j}(t),\quad \text{ with } \beta_{j}
  = \frac{\int_{\Omega_{2}} \phi_{j}(s) g(s)\,ds}{\sigma_{j}}.
\end{equation}
If we truncate this series after $\ell$ terms and use the noisy right hand side
$g^{\delta}$ instead of $g$, then we obtain the TSVE solution to \eqref{Axgd}
defined by
\begin{equation}\label{TSVDsol1D} x_{\ell}(t)=\sum_{j=1}^{\ell}
  \beta^{\delta}_{j} \psi_{j}(t),\quad \text{ with } \beta^{\delta}_{j}=
  \frac{\int_{\Omega_{2}} \phi_{j}(s) g^{\delta}(s)\,ds}{\sigma_{j}}.
\end{equation}
The truncation parameter $\ell$ can be chosen as needed.

In the following lemma, we link the projection of the error onto the space
spanned by the $\phi_{i}(s)$ to the norm of the error. 
\begin{lemma}
Let $n(s)=g(s)-g^{\delta}(s)$ with $ \norm{n(s)} \leq \delta$. Then,
\begin{equation}\label{lemma3.1}
\sum_{i=1}^{\infty}\left(\int_{\Omega_{2}}
      \phi_{i}(s) n(s)\,ds\right)^{2}\leq \delta^{2}, 
\end{equation}      
      where $\phi_i(s)$ are orthonormal basis functions. 
\end{lemma} 
\begin{proof}
  Using the basis functions $\phi_i(s)$, $n(s)$ can be represented
  as
  \begin{align*}
    n(s)=\sum\limits_{j=1}^{\infty} \gamma_{j}\phi_{j}(s)+\phi^{\perp}(s),
  \end{align*}
for certain coefficients $\gamma_j$, and
where $\phi^{\perp}(s)$ is orthogonal to all functions $\phi_j(s)$.
  Then,
 \begin{align*}
   \int_{\Omega_{2}} \phi_{i}(s) n(s)\,ds&=\int_{\Omega_{2}}
   \phi_{i}(s) \left(\sum\limits_{j=1}^{\infty} \gamma_{j}\phi_{j}(s)+\phi^{\perp}(s) \right)\,ds.\\
   \intertext{The orthogonality of the basis functions $\phi_{j}$ allows us to
   simplify the above expression to}
   \int_{\Omega_{2}}
   \phi_{i}(s) n(s)\,ds&=\gamma_{i}.
 \end{align*}
The same argument can be used to show that
$$\sum\limits_{j=1}^{\infty} \gamma_{j}^{2} \leq \norm{n(s)}^{2} \leq
 \delta^{2}. $$ 
 Combining these results shows \eqref{lemma3.1}. \hfill$\square$
\end{proof}

We will now use the previous lemma to given an upper bound for the error of the
solution obtained with the TSVE regularization method.   
\begin{lemma}
  \label{lemma:tsvd_error}
  Let $x(t)$ and $x_{\ell}(t)$ be the exact solution and the TSVE regularized
  solutions given by \eqref{truesol1D} and \eqref{TSVDsol1D}, respectively.
  Assume the kernel $\kappa(s,t)$ has finite rank $r$. Then,
  	\begin{equation}\label{lemmatsve}
     	 \norm{x(t)-x_{\ell}(t)} \leq \left(\frac{\delta^{2}}{\sigma_{\ell}^{2}} 			+\sum\limits_{i=\ell+1}^{r} \beta_{i}^{2}\right)^{1/2}.
    \end{equation}
\end{lemma}

\begin{proof} We will rely on the expansion of the
  solution in the space spanned by the functions $\psi_{i}(t)$.
We have
  \begin{align*}
    \norm{x(t)-x_{\ell}(t)}^{2}&=\norm{\sum\limits_{i=1}^{r} \beta_{i}\psi_{i}(t)
      -\sum\limits_{i=1}^{\ell} \beta^{\delta}_{i} \psi_{i}(t)}^{2}.\\
    \intertext{Using \eqref{truesol1D} and \eqref{TSVDsol1D} this simplifies to}
      \norm{x(t)-x_{\ell}(t)}^{2}&=\norm{\sum\limits_{i=1}^{\ell} \frac{\int_{\Omega_{2}}
        \phi_{i}(s) \left(g(s)-g^{\delta}(s)\right)\,ds}{\sigma_{i}}
        \psi_{i}(t)+\sum\limits_{i=\ell+1}^{r} \beta_{i}
        \psi_{i}(t)}^{2}\\
      &\leq \norm{\sum\limits_{i=1}^{\ell} \frac{\int_{\Omega_{2}} \phi_{i}(s) 
        \left(g(s)-g^{\delta}(s)\right)\,ds}{\sigma_{i}} \psi_{i}(t)}^{2}
        + \norm{\sum\limits_{i=\ell+1}^{r} \beta_{i} \psi_{i}(t)}^{2}.\\
    \intertext{The orthonormality of the basis functions $\psi_{i}$ allows us to simplify the above inequality to}   
      \norm{x(t)-x_{\ell}(t)}^{2}&\leq \sum\limits_{i=1}^{\ell} \left(\frac{\int_{\Omega_{2}} \phi_{i}(s) 
        \left(g(s)-g^{\delta}(s)\right)\,ds}{\sigma_{i}}\right)^{2}+\sum\limits_{i=\ell+1}^{r} \beta_{i}^{2}.\\
        \intertext{Using Lemma \ref{lemma3.1} and the fact the singular values are in non-increasing order gives } 
     \norm{x(t)-x_{\ell}(t)}^{2}   &\leq \frac{\delta^{2}}{\sigma_{\ell}^{2}} +\sum\limits_{i=\ell+1}^{r} \beta_{i}^{2}.
  \end{align*} \hfill$\square$
\end{proof}

\begin{figure}[t]
  \centering
  \begin{tikzpicture}
    \begin{axis}[
      width=5in, %
      height=2.2in, %
      ybar, %
      ymax=2.5, %
      legend style={%
        at={(0.95,0.95)}, %
        anchor=north east, %
        cells={anchor=west}, %
      }, %
      symbolic x coords={Baart, Foxgood, Gravity, Shaw, Wing}, %
      xtick=data, %
      ]

      \addplot[fill=SPECorange] coordinates {%
        (Baart,0.209830752166255) %
        (Foxgood,0.017964227827709) %
        (Gravity,0.036927582968480) %
        (Shaw, 0.086674859374278) %
        (Wing, 0.348049438529407) %
      };
      \addplot[fill=SPECred] coordinates {%
        (Baart,0.250507045532956) %
        (Foxgood,1.901536166479170) %
        (Gravity,1.874650936624351) %
        (Shaw, 1.518125686327245) %
        (Wing, 0.423628546907915) %
      };
      \legend{left-hand side of \eqref{lemmatsve},right-hand side of \eqref{lemmatsve}}
    \end{axis}
  \end{tikzpicture}
  \caption{Behavior of the bound \eqref{lemmatsve} for the examples Baart,
    Foxgood, Gravity, Shaw, and Wing.}\label{fig:1}
\end{figure}
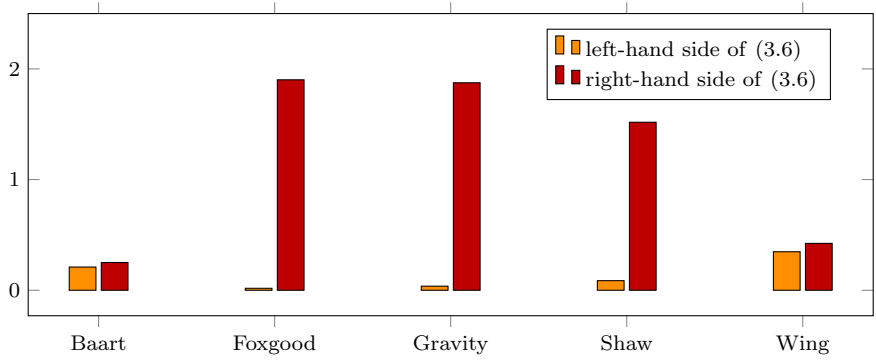

Lemma~\ref{lemma:tsvd_error} provides a justification for chosing $\ell$ such that
\begin{align*}
  \sigma_{\ell}\leq \eta \delta, 
\end{align*} %
with $\eta$ being a small constant greater than $1$. If $\sigma_{\ell}=\delta$,
then the bound from Lemma~\ref{lemma:tsvd_error} is at least $1$. Choosing
$\eta$ larger means that $\frac{\delta^{2}}{\sigma_{\ell}}$ will be
smaller. However, there is a trade-off, since additional $\beta_{i}$ have to be
included in the bound. Generally, choosing $\eta$ between to $2$ and $5$ is
reasonable. Figure \ref{fig:1} illustrates the behavior of the bound
\eqref{lemmatsve} for some numerical examples.

Ill-posed problems based on one-dimensional integral equations are arguably less
challenging than 2D-problems. Thus, consider the two-dimensional Fredholm
integral equations of the first kind,
\begin{equation}\label{integral2d}
  \int_{\Omega_{1}} \kappa(s_{1},s_{2},t_{1},t_{2}) x(t_{1},t_{2})\,dt_{1}dt_{2}
  = 
  g^{\delta}(s_{1},s_{2}),\qquad (s_{1},s_{2})\in\Omega_{2}.
\end{equation}
We employ Chebfun for the numerical solution of the ill-posed
problems. Hence, we are limited by Chebfun's capabilities to deal with 
higher-dimensional functions. A kernel that can be separated into a product of two
functions, i.e.,
$\kappa(s_{1},s_{2},t_{1},t_{2})=\kappa_{1}(s_{1},t_{1})\times\kappa_{2}(s_{2},t_{2}),$
can be handled by Chebfun. The kernel that models Gaussian blur provides an example and
will be used in a numerical illustration. Let the kernel be given by
\begin{equation}\label{kernel2d}
  \kappa(s_{1},s_{2},t_{1},t_{2})=\sum\limits_{i=1}^{r_{1}} \sigma_{i}
  \phi_{i}^{(1)}(s_{1})\psi_{i}^{(1)}(t_{1})\ 
  \sum\limits_{j=1}^{r_{2}} \mu_{j} \phi_{j}^{(2)}(s_{2})\psi_{j}^{(2)}(t_{2}),
\end{equation} 
where both the $\sigma_i$ and $\mu_j$ denote singular values,
and let the solution be of the form 
\begin{equation}\label{solexp2d}
  x(t_{1},t_{2})= \sum\limits_{k=1}^{r_{1}} \sum\limits_{\ell=1}^{r_{2}} \beta_{k \ell}  \psi_{k}^{(1)}(t_{1})\psi_{\ell}^{(2)}(t_{2}).
\end{equation}
By substituting \eqref{kernel2d} and \eqref{solexp2d} into \eqref{integral2d},
and using the orthonormality of the basis functions, we get
\begin{align*}
  \sigma_{i} \mu_{j} \beta_{ij} \phi_{i}^{(1)}(s_{1})  \phi_{j}^{(2)}(s_{2}) =g^{\delta}(s_{1},s_{2}). 
\end{align*}
We further probe the equation with
$\phi_{k}^{(1)}(s_{1}) \phi_{\ell}^{(2)}(s_{2})$ for all $k$ and $\ell$ and use
the orthonormality of the basis functions to obtain
\begin{equation}\label{coeff2d}
  \beta_{ij}  =\frac{\int_{\Omega_{2}} g^{\delta}(s_{1},s_{2})
    \phi_{i}^{(1)}(s_{1})\phi_{j}^{(2)}(s_{2})\,ds_{1} ds_{2}}{\sigma_{i} \mu_{j}}.
\end{equation} This allows us to implement the solution algorithm using at most functions of
three variables and, thus, not exceeding Chebfun3's capabilities \cite{HasTre17}.

In order to solve problems in two space-dimensions with a non-separable kernel, we would 
need Chebfun4, which currently is not available.

\section{Tikhonov regularization}\label{sec:tik}

For Tikhonov regularization, instead of solving \eqref{Axg} exactly, we solve
the functional minimization problem
\begin{equation}\label{tikmineq}
  \min_{x \in {H_{1}}}\left\{ \norm{Ax-g^{\delta}}^{2}+\lambda^{2}\norm{x}^{2}\right\},
\end{equation}
where $\lambda$ is a fixed positive number. Using the definition of $L^{2}$-norm,
equation \eqref{tikmineq} can be written as
\begin{equation}\label{3.2}
  \underset{x\in {H_{1}}}{\text{min}}\left\{\int_{\Omega_{2}}
  \bigabs{Ax-g^{\delta}}^{2}\,ds 
  + \lambda^{2} \int_{\Omega_{1}} \abs{x}^{2}\,dt\right\}.
\end{equation}
By substituting \eqref{kernel} and \eqref{solexpress} into \eqref{3.2}, and by
using the orthonormality of the basis functions, we obtain
\begin{align*}
  \min_{x(t)\in {H_{1}}} \sum\limits_{i=1}^{\infty} \left(\sigma^{2}_{i}
  \beta^{2}_{i} -2\sigma_{i} \beta_{i} \int_{\Omega_{2}} {\phi_{i}(s)}
  g^{\delta}(s)\,ds
  + \lambda^{2}\beta^{2}_{i}\right) +\int_{\Omega_{2}}\abs{g^{\delta}(s)}^{2}\,ds. 
\end{align*}
Thus, we can compute the solution as
\begin{equation}\label{solTik1D}
  x_{\lambda}(t)=\sum\limits_{j=1}^{\infty} \beta^{(\lambda)}_{j} \psi_{j}(t),\quad 
  \text{ with } \beta^{(\lambda)}_{j}= \frac{\sigma_{i} 
    \int_{\Omega_{2}} {\phi_{i}(s)} g^{\delta}(s)\,ds}{(\sigma^{2}_{i} +\lambda^{2})}.
\end{equation}

For the two-dimensional case, instead of solving \eqref{integral2d} exactly, we
solve
\begin{equation}\label{Tikeq2d}
\begin{split}
  \min_{x(t_1,t_2)\in {H_1}}\bigg\{\bignorm{\int_{\Omega_1}
    \kappa(s_1,s_2,t_1,t_2)x(t_1,t_2) \,dt_1\,dt_2 - g^{\delta}(s_1,s_2) }^{2} +\\
  \lambda^{2} \norm{ x(t_1,t_2)}^{2}\bigg\}.
\end{split}
\end{equation} 
By substituting \eqref{kernel2d} and \eqref{solexp2d} into \eqref{Tikeq2d}, and
using the orthonormality of the basis functions, we get
\begin{equation}
\begin{split}
  \min_{x(t_{1},t_{2})\in {H_{1}}}\sum_{i=1}^{r_{1}} \sum_{j=1}^{r_{2}} \bigg(
  \beta^{2}_{ij} \sigma^{2}_{i} \mu^{2}_{j} -2 \beta_{ij}\sigma_{i} \mu_{j}
  \int_{\Omega_{2}}\phi_{i}^{(1)}(s_{1})\phi_{j}^{(2)}(s_{2})\
  g^{\delta}(s_{1},s_{2})\,ds_{1}\,ds_{2} +\\ \lambda^{2} \beta^{2}_{ij} \bigg)
  +\int_{\Omega_{2}} \abs{ g^{\delta}(s_{1},s_{2})}^{2} \,ds_{1}\,ds_{2},
\end{split}
\end{equation}
and we can compute the solution by
\begin{equation}\label{solTik2D}
\begin{split}
  x_{\lambda}(t_{1},t_{2})&= \sum_{k=1}^{r_{1}} \sum_{\ell=1}^{r_{2}} \beta_{k
    \ell} \psi_{k}^{(1)}(t_{1})\psi_{\ell}^{(2)}(t_{2}),\\ \text{ with }
  \beta_{k\ell}&= \frac{\sigma_{k}
    \mu_{\ell}\int_{\Omega_{2}}\phi_{k}^{(1)}(s_{1})\phi_{\ell}^{(2)}(s_{2})\
    g^{\delta}(s_{1},s_{2})\,ds_{1}\,ds_{2}}{\sigma^{2}_{k} \mu^{2}_{\ell} +
    \lambda^{2}} .
\end{split}
\end{equation}
\section{Numerical experiments}\label{sec:numerical_experiments}
In this section we illustrate the performance of the methods described in
Sections~\ref{sec:tsve} and~\ref{sec:tik} by reporting some numerical results.

We first consider five test problems in one space-dimension. These problems are
from Regularization Tools by Hansen \cite{PCH4}. This will be followed by
applying the methods to a 2-D problem. All computations
were carried out in MATLAB R2017a with about 15 significant
decimal digits running on a laptop computer with core CPU Intel(R) Core(TM)i7-7Y75 @1.30GHz 1.60GHz processor with
16GB of RAM.

Each test problem from Regularization Tools by Hansen \cite{PCH4} provides us
with an integral equation of the form \eqref{inteq}. These problems are discretized
by a Nystr\"{o}m method or a Galerkin method with orthogonal test and trial functions
to give a linear system of equations $\tilde{A}\bx=\bg$, where
$\tilde{A}\in\R^{n\times n}$ is the discretized integral operator, $\bx\in\R^n$
is a discretization of the exact solution $x_{\text{exact}}$, and $\bg\in\R^n$ is 
the corresponding error-free right-hand side vector. We generate the 
error-contaminated vector $\bg^{\delta}\in\R^n$ according to 
\begin{equation*} 
 \bg^{\delta}=\bg+\alpha\frac{\norm{\bg}_2}{\norm{\be}_2}\be, 
\end{equation*}
where $\be\in\R^{n}$ is a random vector whose entries are from a normal
distribution with mean zero and variance one.  
In our methods,we use the Matlab package Chebfun \cite{chebfun} to represents the kernel
$\kappa(s,t)$, a function $g(s)$ that represents the error-free right-hand side,
and the desired solution $x(t)$. We define the error-contaminated function
$g^{\delta}(s)$ by
\begin{equation*} 
g^{\delta}(s)=g(s)+\alpha \frac{\norm{g(s)}}{\norm{F(s)}}F(s), 
\end{equation*}
where $F(s)$ is a smooth Chebfun function with maximum frequency about
$2\pi/\vartheta$ and standard normal distribution $N(0,1)$ at each point and
$\alpha$ is the noise level.  In the computed examples, we let
$\vartheta=10^{-2}$. This is Chebfun's analogue to noise. Alternatively, we can
use the discretized right-hand side from regularization tools \cite{PCH4}.

The discrepancy principle is used to determine the truncation parameter~$\ell$
in \eqref{TSVDsol1D} in the TSVE method, and the Tikhonov regularization parameter 
$\lambda$ in \eqref{solTik1D}. The discrepancy principle prescribes that the truncation index
$\ell$ be chosen as small as possible so that the solution $x_{\ell}(t)$ of
\eqref{TSVDsol1D} satisfies
\begin{equation*}
\norm{\int_{\Omega_{1}} \kappa(s,t) x_{\ell}(t)\,dt - g^{\delta}(s)} \leq \eta \delta,
\end{equation*}
where $\eta \geq 1$ is a user-supplied constant independent of $\delta$. The
discrepancy principle, when used with Tikhonov regularization, prescribes
that the regularization parameter $\lambda > 0$ be chosen so that the solution
$x_{\lambda}$ of \eqref{tikmineq} satisfies
\begin{equation*}
  \norm{\int_{\Omega_{1}} \kappa(s,t) x_{\lambda}(t)\,dt - g^{\delta}(s)} = \eta \delta.
\end{equation*}
We use the MATLAB function \texttt{fminbnd} to find the $\lambda$-value and we
let $\eta=1$.

One of the five test problems that we are interested in solving is Baart. This
example is a Fredholm integral equation of the first kind \eqref{inteq} with
$ \kappa(s, t) = \exp(s\, \cos(t))$, $g(s) = 2\,\sinh(s)/s$, and solution
$x(t) = \sin(t)$, where $\Omega_1 = [0, \pi]$ and $\Omega_{2} = [0, \pi/2]$.

We will compute approximate solutions of $x(t) = \sin(t)$ by applying TSVE and
Tikhonov regularization with Chebfun.  These approximate solutions
$x_{\ell}(t)$ and $x_{\lambda}(t)$ can be computed by using the formulas
\eqref{TSVDsol1D} and \eqref{solTik1D}, respectively.
\input{Baart_figure5_2}
Fig.\ \ref{fig:5.1}(a) displays the kernel $\kappa(s, t)$ of the Baart example. The
right-hand side function $g(s)$ and the corresponding error-contaminated
function $g^{\delta}(s)$ are illustrated in Fig.\ \ref{fig:5.1}(b), where the
level noise is $10^{-2}$.  Fig.\ \ref{fig:5.1}(c) depicts the exact solution and
the computed approximate solutions determined by TSVE and Tikhonov
regularization with Chebfun. The latter figure shows that our methods give good
approximation solutions of the exact solution.

Next, we will apply our methods to several different examples.  Moreover, we
will compare the methods with standard TSVD and Tikhonov regularization in
discretized setting. The quality of the computed approximate solutions is
measured with the relative error norm
\begin{equation*}
  RE:=\frac{\norm{x_{\text{method}}-x}_{*}}{\norm{x}_{*}} ,
\end{equation*}
where $\norm{\cdot}_{*}$ denotes the Euclidean vector norm
$(\frac{1}{n}\sum_{i=1}^{n}x_{i}^2)^{1/2}$ if $x$ is a vector, or
the $L^{2}$-norm if $x$ is a function.
\begin{table}[tbp]
\begin{center}
\renewcommand{\arraystretch}{1.3}
\caption{Comparison of TSVE with Chebfun and for the discretized problem.}
\footnotesize
\begin{tabular}{ccccccc}
\toprule
 \multirow{2}{3em}{Noise level}& \multirow{2}{3em}{Example}&\multicolumn{3}{c}{discretized} &\multicolumn{2}{c}{with Chebfun}\\
 && n  & RE&CPU  & RE&CPU\\ 
 \midrule
\multirow{5}{3em}{$10^{-3}$} &\texttt{baart} &$1376$&$1.1479\cdot 10^{-1}$& $2.1327\cdot 10^{0}$& $1.1479\cdot 10^{-1}$ &$4.8078\cdot 10^{-1}$\\ 
 & \texttt{foxgood} &$332$&$9.8663\cdot 10^{-3}$& $4.6138\cdot 10^{-2}$& $9.8653\cdot 10^{-3}$ &$4.0253\cdot 10^{0}$\\ 
 & \texttt{gravity}&$209$&$1.9936\cdot 10^{-2}$& $3.2203\cdot 10^{-2}$& $1.9939\cdot 10^{-2}$ &$2.8642\cdot 10^{0}$\\ 
 &  \texttt{shaw} &$7$&$3.9299\cdot 10^{-2}$& $1.3183\cdot 10^{-3}$& $4.1005\cdot 10^{-2}$ &$1.6327\cdot 10^{0}$\\ 
&  \texttt{wing} &$822$&$6.0280\cdot 10^{-1}$& $2.9813\cdot 10^{-1}$& $6.0280\cdot 10^{-1}$ &$5.2352\cdot 10^{-1}$ \\ 
\midrule
 \multirow{5}{3em}{$10^{-2}$} &\texttt{baart} &$470$&$1.6644\cdot 10^{-1}$& $7.6005\cdot 10^{-2}$& $1.6644\cdot 10^{-1}$ &$ 1.6913\cdot 10^{-1}$\\ 
 & \texttt{foxgood} &$327$&$3.1572\cdot 10^{-2}$& $2.5878\cdot 10^{-2}$& $3.1575\cdot 10^{-2}$ &$1.8004\cdot 10^{0}$\\  
 &  \texttt{gravity}&$152$&$4.0750\cdot 10^{-2}$& $4.6737\cdot 10^{-3}$& $4.0751\cdot 10^{-2}$ &$1.0127\cdot 10^{0}$\\ 
 &  \texttt{shaw} &$720$&$1.3119\cdot 10^{-1}$& $2.4440\cdot 10^{-1}$& $1.3087\cdot 10^{-1}$ &$1.0146\cdot 10^{0}$ \\  
 & \texttt{wing} &$264$&$6.0280\cdot 10^{-1}$& $1.1490\cdot 10^{-2}$& $6.0280\cdot 10^{-1}$ &$1.1596\cdot 10^{-1}$\\  
\midrule
 \multirow{5}{3em}{$10^{-1}$}& \texttt{baart} &$460$&$3.4643\cdot 10^{-1}$& $6.0992\cdot 10^{-2}$& $3.4644\cdot 10^{-1}$ &$4.7546\cdot 10^{-1}$\\
&  \texttt{foxgood} &$765$&$7.5584\cdot 10^{-2}$& $2.5744\cdot 10^{-1}$& $7.5813\cdot 10^{-2}$ &$1.6939\cdot 10^{0}$\\  
&  \texttt{gravity}&$1730$&$6.6598\cdot 10^{-2}$& $3.6553\cdot 10^{0}$& $6.6607\cdot 10^{-2}$ &$6.9005\cdot 10^{-1}$\\  
&  \texttt{shaw} &$1703$&$1.5246\cdot 10^{-1}$& $2.7470\cdot10^{0}$& $1.5267\cdot 10^{-1}$ &$4.2545\cdot 10^{-1}$\\  
 &  \texttt{wing}&$276$&$6.1568\cdot 10^{-1}$& $2.2018\cdot 10^{-2}$& $6.1542\cdot 10^{-1}$ &$5.0701\cdot 10^{-1}$\\  
\bottomrule
\end{tabular}
\label{tab:comparison of TSVE}
\end{center}
\end{table}
\begin{table}[tbp]
\begin{center}
\renewcommand{\arraystretch}{1.3}
\caption{Comparison of Tikhonov regularization with Chebfun and for the discretized problem.}
\footnotesize
\begin{tabular}{ccccccc}
\toprule
 \multirow{2}{3em}{Noise level}& \multirow{2}{3em}{Example}&\multicolumn{3}{c}{discretized} &\multicolumn{2}{c}{with Chebfun}\\
 && n  & RE&CPU  & RE&CPU\\ 
 \midrule  
\multirow{5}{3em}{$10^{-3}$} &\texttt{baart} &$184$&$1.3228\cdot 10^{-1}$& $6.9910\cdot 10^{-3}$& $1.3220\cdot 10^{-1}$ &$2.8148\cdot 10^{0}$\\ 
 & \texttt{foxgood} &$363$&$1.2252\cdot 10^{-2}$& $5.2424\cdot 10^{-2}$& $1.2250\cdot 10^{-2}$ &$4.5776\cdot 10^{1}$\\ 
 & \texttt{gravity}&$1250$&$1.5306\cdot 10^{-2}$& $1.1072\cdot 10^{0}$& $1.5298\cdot 10^{-2}$ &$5.7426\cdot 10^{0}$\\ 
 &  \texttt{shaw} &$947$&$4.4255\cdot 10^{-2}$& $5.2421\cdot 10^{-1}$& $4.4253\cdot 10^{-2}$ &$7.3928\cdot 10^{0}$\\ 
&  \texttt{wing} &$1353$&$6.0277\cdot 10^{-1}$& $2.0925\cdot 10^{0}$& $6.0277\cdot 10^{-1}$ &$1.0410\cdot 10^{1}$ \\ 
\midrule
 \multirow{5}{3em}{$10^{-2}$} &\texttt{baart} &$1332$&$1.7066\cdot 10^{-1}$& $1.6434\cdot 10^{0}$& $1.7067\cdot 10^{-1}$ &$1.8585\cdot 10^{0}$\\ 
 & \texttt{foxgood} &$1922$&$2.3125\cdot 10^{-2}$& $4.5407\cdot 10^{0}$& $2.3124\cdot 10^{-2}$ &$3.7736\cdot 10^{1}$\\  
 &  \texttt{gravity}&$1527$&$2.8709\cdot 10^{-2}$& $2.2178\cdot 10^{0}$& $2.8708\cdot 10^{-2}$ &$5.3568\cdot 10^{0}$\\ 
 &  \texttt{shaw} &$186$&$1.1000\cdot 10^{-1}$& $ 1.5628\cdot 10^{-2}$& $1.0998\cdot 10^{-1}$ &$5.5254\cdot 10^{0}$ \\  
 & \texttt{wing} &$1232$&$6.0340\cdot 10^{-1}$& $1.4646\cdot 10^{0}$& $6.0340\cdot 10^{-1}$ &$3.2663\cdot 10^{0}$\\  
\midrule  
 \multirow{5}{3em}{$10^{-1}$}& \texttt{baart} &$568$&$ 2.2781\cdot 10^{-1}$& $ 1.3464\cdot 10^{-1}$& $2.2769\cdot 10^{-1}$ &$1.7098\cdot 10^{0}$\\
&  \texttt{foxgood} &$154$&$5.4066\cdot 10^{-2}$& $1.4565\cdot 10^{-2}$& $5.4079\cdot 10^{-2}$ &$3.3878\cdot 10^{1}$\\  
&  \texttt{gravity}&$163$&$8.8483\cdot 10^{-2}$& $1.4838\cdot 10^{-2}$& $8.8507\cdot 10^{-2}$ &$1.8138\cdot 10^{1}$\\  
&  \texttt{shaw} &$862$&$1.6105\cdot 10^{-1}$& $4.3403\cdot 10^{-1}$& $1.6106\cdot 10^{-1}$ &$4.9331\cdot 10^{0}$\\  
 &  \texttt{wing} &$12$&$6.5959\cdot 10^{-1}$& $6.1593\cdot 10^{-3}$& $6.5836\cdot 10^{-1}$ &$7.2559\cdot 10^{0}$\\  
\bottomrule
\end{tabular}
\label{tab:Comparison of Tikhonov regularization}
\end{center}
\end{table}


Tables \ref{tab:comparison of TSVE} and \ref{tab:Comparison of Tikhonov regularization} 
compare the TSVE and Tikhonov regularization methods when used with Chebfun and with 
standard methods for the test problems Baart, Foxgood, Gravity, Shaw, and Wing from
\cite{PCH4}. Three noise levels $\alpha$ are considered. The number of discretization 
points, $n$, which is shown in the third column of the tables, is chosen to be between 
$1$ and $2000$, so that the smallest absolute difference between the relative error of 
the solution for the discretized problem and the relative error of the solution for 
the continuous problem is achieved. Thus, we choose the number of discretization points
$n$ so that the discretized problem gives an approximate solution of about the same 
accuracy as the approximate solution determined with Chebfun. This choice makes a
comparison of the CPU-times required by the methods meaningful. The relative errors 
obtained by applying TSVD and 
Tikhonov regularization in the discretized setting are reported in the fourth column 
of Tables \ref{tab:comparison of TSVE} and \ref{tab:Comparison of Tikhonov regularization},
respectively. The sixth column of the tables shows the relative errors obtained when
applying TSVE and Tikhonov regularization with Chebfun. We also report the CPU times 
in seconds for each method in the fifth and seventh columns of tables. The tables 
show the computed approximate solutions determined by Chebfun-based methods to give 
as accurate approximations of the exact solutions as the approximate solutions 
determined by standard methods for the discretized problems. Moreover, we observe that 
the methods based on Chebfun are competitive time-wise for some problems, while they 
are slower for most problems. The last column of Tables \ref{tab:comparison of TSVE} 
and \ref{tab:Comparison of Tikhonov regularization} shows that applying TSVE with 
Chebfun is faster than applying Tikhonov regularization with Chebfun. This is 
reasonable since the TSVE method does not require the use of a root-finder.
\input{shaw}
\input{baart2}
\input{wing1}
 
The accuracy and the run time for the discretized methods depend on the number of
discretization points $n$; Chebfun-based methods do not depend on $n$. 
Thus, in Figures \ref{fig:5.3}, \ref{fig:5.5}, and \ref{fig:5.5}, we show some graphs 
with the relative accuracy on the vertical axis and run time on the horizontal axis; 
being closer to the origin is better. In the figures we trim some of the outliers when 
some values of $n$ give bad accuracy. The figures show that the accuracy and 
computing time of the implementations with Chebfun are competitive.

Finally, we will consider a Fredholm integral equation of the first kind in two
space-dimensions,
\begin{equation}\label{2d}
  \int_{\Omega} \kappa(s_{1},s_{2},t_{1},t_{2}) x(t_{1},t_{2})\,dt_{1}dt_{2}
  = 
  g^{\delta}(s_{1},s_{2}),\qquad (s_{1},s_{2})\in\Omega,
\end{equation}
where $\Omega=[-1,1]\times[-2,2]$. The kernel models Gaussian blur and is given by
\begin{equation*}
\kappa(s_{1},s_{2},t_{1},t_{2})=\kappa_{1}(s_{1},t_{1})\times\kappa_{1}(s_{2},t_{2}), 
\end{equation*}
with  
\begin{equation*}
\kappa_{1}(s_{1},t_{1})=\frac{e^{-\frac{\left(t_1-s_1\right)^2}{2\sigma^2}}}{\sqrt{2\pi\sigma^2}},
\end{equation*}
where $\sigma$ is the standard deviation of the Gaussian distribution.  The
exact solution $x(t_{1},t_{2})$ will be constructed as a continuous
``image''\footnote{\!\!\!With a continuous ``image'' we mean a mapping from
  $[0,1]\times[0,1]$ to $[0,1]$, where the function value represents a gray
  scale value. Thus, a gray scale value exists for all points continuously and
  not just for discrete points on a grid. The mapping itself is not necessarily
  continuous.}  that we will blur and try to reconstruct.  In our example, we
let $\sigma=0.2$ and construct the exact solution as
\begin{equation*}
  x(t_{1},t_{2})=\left\lbrace  (t_{1},t_{2})\in \Omega \,: -0.5<t_1 <0.2 \, \,
    \text{and}\,    -0.6<t_2 <-0.2         	 \right\rbrace,
\end{equation*}
which is shown in Fig.\ \ref{figure2d}(a). The error-free right-hand side
function is determined by
\begin{equation*}
   g(s_{1},s_{2}):=\int_{\Omega} \kappa(s_{1},s_{2},t_{1},t_{2}) \, x(t_{1},t_{2})\,dt_{1}dt_{2}
\end{equation*}
and the error-contaminated function $ g^{\delta}(s_{1},s_{2})$ in \eqref{2d} is
defined by
\begin{equation*} 
  g^{\delta}(s_{1},s_{2})=  g(s_{1},s_{2})+\alpha \frac{\norm{g(s_{1},s_{2})}}{\norm{F(s_1,s_2)}}F(s_1,s_2), 
\end{equation*}
where $F(s_1,s_2)$ is a smooth Chebfun function in two space-dimensions with maximum 
frequency about $2\pi/\vartheta$ and standard normal distribution $N(0,1)$ at each point
and $\alpha$ is the noise level.  In this problem, we let the noise level and
$\vartheta$ equal $10^{-2}$. Both the error-free right-hand side and the
error-contaminated function are shown in Fig. \ref{figure2d}(b) .

We reconstruct the exact image $x(t_1,t_2)$ by applying the Chebfun-based methods to the
problem.  Similarly as for the problems in one space-dimension, the truncation parameter
$\ell$ in \eqref{solexp2d} and the Tikhonov regularization parameter $\lambda$ in
\eqref{solTik2D} are determined with aid of the discrepancy principle, where
we set $\eta $ to be $10$ in our example. The reconstructed images obtained with
the TSVE and Tikhonov regularization with Chebfun are shown in
Fig. \ref{figure2d}(c) and (d), respectively.  The two reconstructed images are
seen to be of roughly the same quality, with the image determined by Tikhonov 
regularization being slightly less oscillatory, and the computing times for both methods
is comparable: the TSVE with Chebfun required $34.64$ seconds, while Tikhonov
regularization with Chebfun took $24.56$ seconds.
\begin{figure}[t]
	\begin{subfigure}[b]{0.4\textwidth}
		\includegraphics[height=6.5cm, width=6cm]{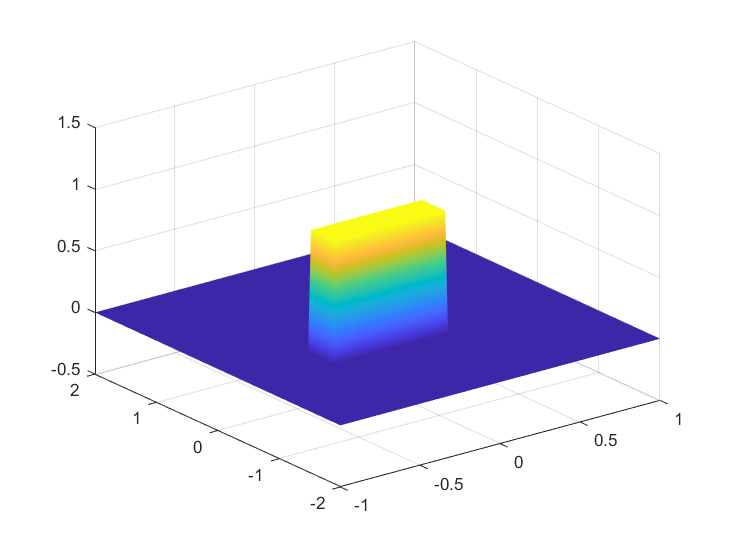}
		 \caption{}
   	\end{subfigure}
  \hfill
  	\begin{subfigure}[b]{0.4\textwidth}
    	\includegraphics[height=6.5cm, width=6cm]{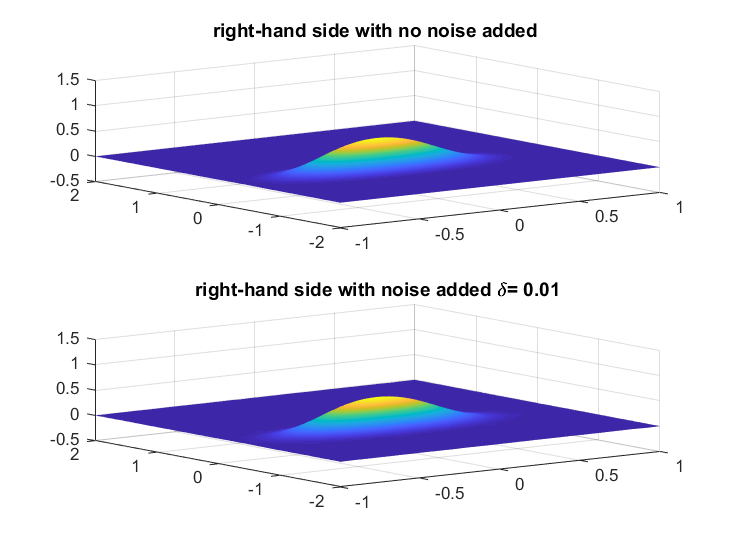}
    	 \caption{}
   	\end{subfigure}
     \hfill
  	\begin{subfigure}[b]{0.4\textwidth}
  		\includegraphics[height=6.5cm, width=6cm]{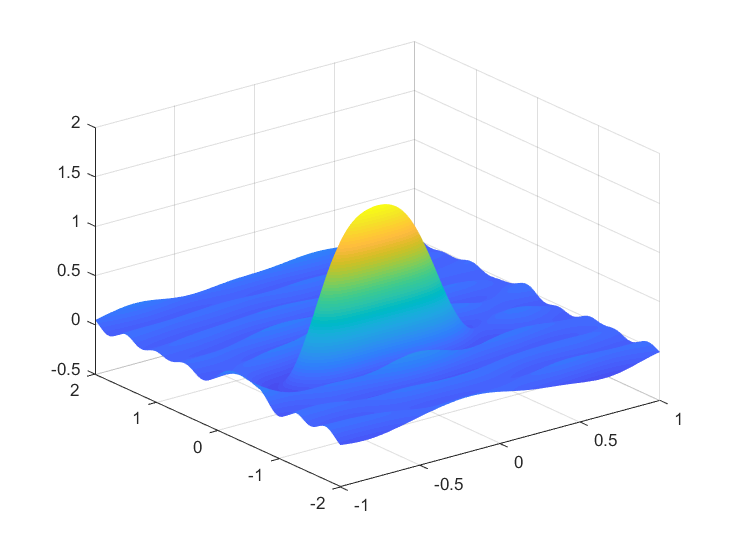}
    	\caption{}
   	\end{subfigure}
   	\hfill
   	  	\begin{subfigure}[b]{0.4\textwidth}
  		\includegraphics[height=6.5cm, width=6cm]{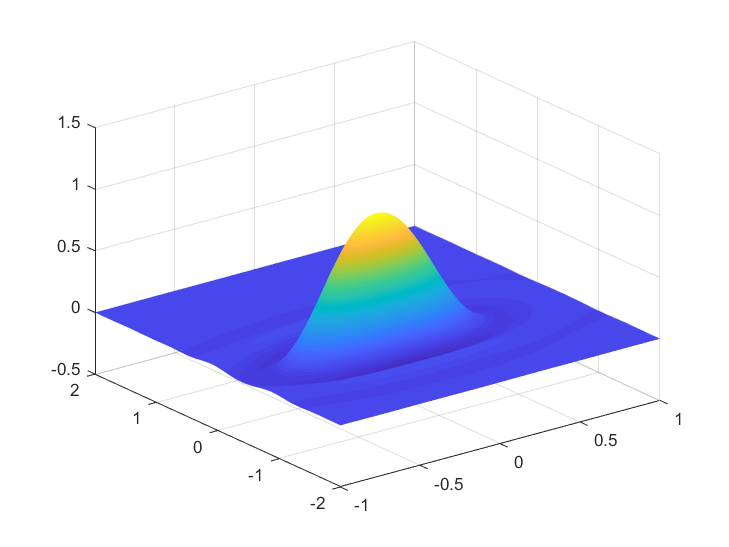}
  		\caption{}
    	\end{subfigure}
    	\hfill
\caption{Gaussian blur example: (a) Exact image, (b) Right hand side, (c)  Reconstructed image by TSVE with Chebfun, (d) Reconstructed image by Tikhonov regularization with Chebfun.}\label{figure2d}
\end{figure}


\section{Conclusion}\label{sec:concl}
The computed results illustrate the feasibility of using Chebfun to solve
linear discrete ill-posed problems and in this way carry out computations in a
fashion that is closer to the spirit of the analysis of ill-posed problems found,
e.g., in \cite{Eng2000}. The accuracy and timings of the implementations with 
Chebfun are competitive.

In the future further extensions to Chebfun including the treatment of functions
of four or six variables will allow the application of the Chebfun-based approach 
discussed in this paper to the solution of linear ill-posed problems in two and
three space-dimensions. It would be interesting to see if the observations made 
here carry over to these classes of problems.

\section*{Acknowledgments}\label{sec:ack}
The authors are grateful for enlightening discussions with Behnam Hashemi
(Shiraz University of Technology) about Chebfun and Chebfun3 in particular. We
hope that this paper can serve as a motivation for the extension of Chebfun to
four and higher dimensional functions.

We also would like to thank Richard Mika\"el Slevinsky (University of Manitoba)
for first pointing out to us the link between adaptive cross approximation and
singular value expansions used in Chebfun2/3.

\bibliographystyle{siam}
\bibliography{longstrings,bib}

\end{document}